\documentclass[11pt]{amsart}
\usepackage{geometry,graphicx}
\usepackage{times}
\usepackage{amsmath,amsfonts,amssymb,latexsym}
\usepackage[kerning=true]{microtype} % babel=true si en franÃ§ais
\usepackage{graphicx}
\usepackage{wrapfig}

\usepackage{stmaryrd}

\theoremstyle{theorem}

\newtheorem{theorem}{Theorem}

\newtheorem{lemma}{Lemma}

\theoremstyle{definition}

\newcommand{\vol}{{\rm vol}}

\newcommand{\R}{{\mathbb R}}

\newcommand{\Z}{{\mathbb Z}}

\newcommand{\T}{{\mathbb T}}

\numberwithin{proposition}{section}
\numberwithin{equation}{section}
\numberwithin{corollary}{section}
\numberwithin{definition}{section}
\numberwithin{lemma}{section}
\numberwithin{claim}{section}
\numberwithin{remark}{section}

\title[Homologically independent loops on tori]{Length product of homologically independent loops for tori}

\author{F.~Balacheff and S.~Karam}

\thanks{This work acknowledges support by grants ANR CEMPI (ANR-11-LABX-0007-01) and ANR Finsler (ANR-12-BS01-0009-02)}

\address{F. Balacheff, Laboratoire Paul Painlev\'e, Bat. M2, Universit\'e des Sciences et Technologies,
59 655 Villeneuve d'Ascq, France.}

\email{florent.balacheff@math.univ-lille1.fr}

\address{S. Karam, Laboratoire Paul Painlev\'e, Bat. M2, Universit\'e des Sciences et Technologies,
59 655 Villeneuve d'Ascq, France.}

\email{steve.karam@math.univ-lille1.fr}



\begin{document}%
%%%%%%%%%%%%%%%%%
%\clearpage

\begin{abstract}
We prove that any Riemannian torus of dimension $m$ with unit volume admits $m$ homologically independent closed geodesics whose length product is bounded from above by $m^m$.
\end{abstract}

\maketitle

The goal of this note is to prove the following analog of Minkowski's second theorem.

\begin{theorem}\label{th:tori}
Let $(\T^m,g)$ be a Riemannian torus of dimension $m\geq2$.
There exist $m$ homologically independent closed geodesics $(\gamma_1,\ldots,\gamma_{m})$ whose length product satisfies 
$$
\prod_{i=1}^{m} \ell_g(\gamma_i)\leq  m^m \cdot \vol(\T^m,g).
$$
\end{theorem}
%This is related to a question raised by M.~Gromov, see \cite[p.75, Some open questions]{Gro83}. \\

Given a closed manifold $M$ of dimension $m$ and a class $\zeta \neq 0 \in H^1(M;\Z_2)$ set
$$
L(\zeta):=\inf \{\ell_g(\gamma) \mid \gamma \,  \, \text{is a closed curve with} \, \, \langle\zeta,[\gamma]\rangle\neq 0\}
$$
where $[\gamma]$ denotes the homology class in $H_1(M;\Z_2)$ corresponding to the curve $\gamma$ and $\langle\cdot,\cdot\rangle$ the pairing between $\Z_2$-cohomology and homology.
The central result of this note is the following statement  from which Theorem \ref{th:tori} can be deduced.

\begin{theorem}\label{th:Minkowski2}
Let $(M,g)$ be a closed Riemannian manifold of dimension $m\geq 1$ and suppose that there exist (not necessarily distinct) cohomology classes $\zeta_1,\ldots,\zeta_m$ in $H^1(M;\Z_2)$ whose cup product $\zeta_1 \cup \ldots \cup \zeta_m \neq 0$ in $H^m(M;\Z_2)$. 
Then
$$
 \prod_{i=1}^m L(\zeta_i)\leq m^m \cdot \vol \, (M,g).
$$
\end{theorem}

In order to show this result we follow the approach by L.~Guth \cite{Gut10} involving nearly minimal hypersurfaces in his alternative proof of Gromov isosystolic inequality \cite{Gro83} in the special case of manifolds whose $\Z_2$-cohomology has maximal cup-length. \\

Theorem \ref{th:tori} can be deduced from this result as follows. 
 Choose a sequence of $\Z_2$-homologically independent closed geodesics $(\gamma_1,\ldots,\gamma_{m})$ in $(\T^m,g)$ corresponding to the $m$ first successive minima, that is
$$
\ell_g(\gamma_k)=\min \{ \lambda \mid \text{there exist} \, \, k  \, \, \Z_2\text{-homologically independent closed curves of length at most} \, \, \lambda\}.
$$
The dual basis $(\zeta_1,\ldots,\zeta_m)$ to the basis $([\gamma_1],\ldots,[\gamma_m])$ satisfies the condition $\zeta_1 \cup \ldots \cup \zeta_m \neq 0$ and that $L(\zeta_k)=\ell_g(\gamma_k)$ which prove Theorem \ref{th:tori}.\\

The rest of this note is devoted to the proof of Theorem \ref{th:Minkowski2}.
By a strictly increasing sequence of closed submanifolds of $M$ we mean a sequence $Z_0\subset Z_1 \subset \ldots \subset Z_{m-1}\subset Z_m=M$ of closed manifolds $Z_i$ of dimension $i$ for $i=0,\ldots,m$. In particular $Z_0$ is a finite collection of points of $M$ and $Z_{m-1}$ an hypersurface.
Given such a sequence and $m$ positive numbers $R_1,\ldots,R_m$ we define another sequence of subsets $D_1\subset \ldots \subset D_m$ by induction as follows: 
$$
D_1:=\{z \in Z_1 \mid d_g(z,Z_0)\leq R_1\} \subset Z_1,
$$
and for $k=2,\ldots,m$
$$
D_k:=\{z \in Z_k \mid d_g(z,D_{k-1})\leq R_k\} \subset Z_k.
$$

Fix $\delta >0$. We will prove by induction that there exists a strictly increasing sequence $Z_0=\{z_0\} \subset Z_1 \subset \ldots \subset Z_{m-1}\subset Z_m=M$ of closed submanifolds of $M$ such that
\begin{itemize}
\item the homology class $[Z_{i-1}] \in H_{i}(Z_i;\Z_2)$ is the Poincar\'e dual of the restriction  $\zeta'_i:={\zeta_i}_{\mid Z_i}$ ;
\item for any sequence $\{R_i\}_{i=1}^m$ of positive numbers such that $2 \sum_{k=1}^i R_k< L(\zeta'_i)$ for $i=1,\ldots,m$ then
$$
\vol \, D_m\geq 2^m \, \prod_{i=1}^m R_i-O(\delta).
$$
\end{itemize}
As $L(\zeta_i)\leq L(\zeta'_i)$ (with equality at least for $i=m$), ordering $\zeta_1,\ldots,\zeta_m$ such that $L(\zeta_1)\leq\ldots\leq L(\zeta_m)$, taking $R_i\to\frac{L(\zeta_i)}{2m}^-$ and then letting $\delta \to 0$ in the above inequality implies Theorem \ref{th:Minkowski2}.\\

 The case $m=1$ is trivial, so suppose that $m>1$ and that the statement is proved for dimensions at most $m-1$. Let ${\bf Z_{m-1}} \in H_{m-1}(M;\Z_2)$ be the Poincar\'e dual to $\zeta_m$. We fix  a smooth embedded and closed hypersurface $Z_{m-1}$ which is {\it $\delta$-minimizing} in the homology class ${\bf Z_{m-1}}$: any other smooth hypersurface $Z'$ representing ${\bf Z_{m-1}}$ satisfies $\vol_{m-1} \,Z' \geq \vol_{m-1} \,Z_{m-1}-\delta$.
 
 The restriction to $Z_{m-1}$ of the cohomologic classes $\zeta_1,\ldots,\zeta_{m-1}$ gives a family of cohomologic classes $\zeta''_1,\ldots,\zeta''_{m-1}$ in $H^1(Z_{m-1};\Z_2)$ such that $\zeta''_1 \cup \ldots \cup \zeta''_{m-1} \neq 0$ in $H^{m-1}(Z_{m-1};\Z_2)$. By the induction hypothesis there exists a strictly increasing sequence $Z_0=\{z_0\} \subset Z_1 \subset \ldots \subset Z_{m-2}$ of submanifolds of $Z_{m-1}$ such that
\begin{itemize}
\item for $i\leq m-1$ the homology class of $Z_{i-1}$ is the Poincar\'e dual of the restriction of $\zeta''_i$ to $Z_i$, which coincides with $\zeta'_i$ defined as the restriction of $\zeta_i$ to $Z_i$ ;
\item for any sequence $\{R_i\}_{i=1}^{m-1}$ of positive numbers satisfying $2\sum_{k=1}^i R_k< L(\zeta'_i)$ for $i\leq m-1$
$$
\vol_{m-1} \, D_{m-1}\geq 2^{m-1} \, \prod_{k=1}^{m-1} R_i-O(\delta).
$$
\end{itemize}

Now fix a sequence $\{R_i\}_{i=1}^m$ of positive numbers such that $2\sum_{k=1}^i R_k< L(\zeta'_i)$ for $i=1,\ldots,m$. We will need the following Lemma:

\begin{lemma}\label{lem:CFL}
Let $c$ be a $1$-cycle in $D_m$. Then there exist loops $\gamma_1,\ldots,\gamma_k \subset D_m$ with $l_g(\gamma_i)<L(\zeta'_m)$ for $i=1,\ldots,k$ and such that $c$ is homologous to the $1$-cycle $\gamma_1+\ldots+\gamma_k$.
\end{lemma}

\begin{proof}[Proof of the Lemma.]
We proceed as in the Curve Factoring Lemma (see  \cite{Gut10}). Just observe that any point of $D_m$ can be connected to $z_0$ through a path in $D_m$ of length at most $\sum_{k=1}^m R_k$.
\end{proof}

Using this Lemma, we can prove the following analog of the version of Stability Lemma due to Nakamura \cite{Nak13}.

\begin{lemma}
For any $r\leq R_m$
$$
\vol_{m-1} \{x \in M \mid d_g(x,D_{m-1})=r \}\geq 2 \, ( \vol_{m-1} \{x \in Z_{m-1}\mid d_g(x,D_{m-1})\leq r\}-\delta).
$$
\end{lemma}

\begin{proof}[Proof of the Lemma.] First note we only have to prove the inequality for $r=R_m$ which writes as follows:
$$
\vol_{m-1} \{x \in M \mid d_g(x,D_{m-1})=R_m \}\geq 2 \, ( \vol_{m-1} (Z_{m-1}\cap D_m)-\delta).
$$
We argue as in the proof of the Stability Lemma in \cite{Nak13}. First note that 
$$
[Z_{m-1}\cap D_m]=0 \in H_{m-1}(D_m,\partial D_m;\Z_2).
$$
For this recall the following argumentation due to \cite{Gut10}. If this is not the case, by the Poincar\'e-Lefschetz duality the cycle $Z_{m-1}\cap D_m$ has a non-zero algebraic intersection number with an absolute cycle $c$ in $D_m$. Using Lemma \ref{lem:CFL} we find a finite number of loops $\gamma_1,\ldots,\gamma_k \subset D_m$ with $l_g(\gamma_i)<L(\zeta_m)$ for $i=1,\ldots,k$ and such that $c$ is homologous to the $1$-cycle $\gamma_1+\ldots+\gamma_k$. But this implies that for some $i=1,\ldots,k$ the intersection with $Z_{m-1}$ is not zero which gives a contradiction with the definition of $L(\zeta_m)$.
Now the proof proceeds {\it mutatis mutandis} as in \cite{Nak13}.
\end{proof}

Then by the coarea formula
\begin{eqnarray*}
\vol_m  \, D_m &=&  \int_0^{R_m} \vol_{m-1} \{x \in M \mid d_g(x,D_{m-1})=r \} dr\\
&\geq&  \int_0^{R_m} 2 \, ( \vol_{m-1} \{x \in Z_{m-1}\mid d_g(x,D_{m-1})\leq r\}-\delta)dr\\
& \geq &  \int_0^{R_m} 2 \, ( \vol_{m-1} D_{m-1}-\delta)dr\\
&\geq& 2 R_m \, ( \vol_{m-1} D_{m-1}-\delta)
\end{eqnarray*}
which proves the assertion.

\end{document}